\documentclass[12pt]{amsart}
\usepackage{amsmath}
\usepackage{amssymb}
\usepackage{amsfonts}
\usepackage{textcomp}
\usepackage{amsthm}
\usepackage{mathrsfs}
\usepackage[latin1]{inputenc}
\usepackage[all]{xy}
\usepackage{hyperref}
\usepackage{url}
\usepackage{graphicx}
\usepackage{tikz-cd}
\usepackage[paperheight=11in,paperwidth=8.5in,left=1in,top=1.25in,right=1in,bottom=1.25in,]{geometry}

\newcommand{\I}{\mathrm{I}}
\newcommand\II{\mathrm{I}\!\mathrm{I}}
\newcommand\III{\mathrm{I}\!\mathrm{I}\!\mathrm{I}}

\newtheorem{theorem}{Theorem}[section]
\newtheorem{lem}{Lemma}[section]
\newtheorem{prop}{Proposition}[section]
\newtheorem{defi}{Definition}[section]
\newtheorem{cor}{Corollary}[section]

\begin{document}
\title[Additive continuity of the renormalized volume under geometric limits]{\textbf{Additive continuity of the renormalized volume under geometric limits}}
\author{Franco Vargas Pallete}
\thanks{Research partially supported by NSF grant DMS-1406301}
\address{Department of Mathematics  \\
 University of California at Berkeley \\
775 Evans Hall \\
Berkeley, CA 94720-3860 \\
U.S.A.}
\email{franco@math.berkeley.edu}

\begin{abstract}
We study the infimum of the renormalized volume for convex-cocompact hyperbolic manifolds, as well as describing how a sequence converging to such values behaves. In particular, we show that the renormalized volume is continuous under the appropriate notion of limit. This result generalizes previous work in the subject.
\end{abstract}
\maketitle
\section{Introduction}
Renormalized volume is a quantity that gives a notion of volume for hyperbolic manifolds which have infinite volume under the classical definition. Its study for convex co-compact hyperbolic 3-manifolds can be found in \cite{KrasnovSchlenker}, while the geometrically finite case which includes rank 1-cusps was developed in \cite{MoroianuGuillarmouRochon}. In this article we will answer to what value is the infimum of the renormalized volume for a given convex-cocompact manifold with incompressible boundary, and how a sequence converging to the infimum behaves. Partial results to this question were given in the acylindrical case by the author in \cite{Vargas16}. The study of local minimum on the acylindrical case was done in parallel by \cite{Moroianu} and \cite{Vargas15}, which \cite{Vargas16} proves to be the infimum. The incompressoble case has be done independently by \cite{BBB17}, where they studied the gradient flow of the renormalized volume. In addition to their finding of the infimum value, we give the notion of additive geometric convergence in order to describe any sequence converging to the infimum.

The strategy to prove our main result is to understand how a sequence converging to the infimum of the renormalized volume could behave, and conclude that such limits are a special kind of geometrically finite hyperbolic manifold. In order to do so, the article is organized as follows: Section \ref{sec:background} deals with the basic concepts and establishes the notation that will be kept in later sections. Sections \ref{sec:limits} gives a description of sequences (up to taking a subsequence) with bounded volume of the convex core. This will serve us to describe our infimum sequence and the appropriate notion of limit. Section \ref{sec:perigeneric} shows that the Dirichlet fundamental polyhedra of a hyperbolic manifold $M$ is generic with respec to $S^2$. While similar results appear in \cite{JorgensenMarden} and \cite{DiazUshijima} (in fact our arguments are extrapolations of theirs), we include this section in a format that suits our needs and for completeness. Section \ref{sec:smalldef} uses the previous section and the classical Dehn-filling argument of \cite{Thurston5} to prove that small perturbations of the infimum limit are also obtained as limits of similar sequences. Section \ref{sec:addlimit} will show the appropriate version of continuity of the renormalized volume. Namely, the limit of renormalized volumes is equal to the sum of the renormalized volumes of each possible limit component. The argument of the section is a reproduction of the continuity argument of \cite{MoroianuGuillarmouRochon} under our notion of limit. Finally, in Section \ref{sec:consequences} we combine all the properties of previous sections to conclude that the infimum limit needs to be a critical point of the renormalized volume. These kind of manifolds are very special (the convex core is totally geodesic), giving the final answer of what value is the infimum of the renormalized volume and how a sequence converging to it behaves.

\textbf{Acknowledgements:} I would like to thank my advisor Ian Agol for his support and guidance during this project. I would also like to thank Kenneth Bromberg for our helpful discussion about our parallel approachs.

\section{Background}\label{sec:background}
A boundary incompressible compact $3$-manifold $M$ will be called \textit{hyperbolizable} if its interior admits a hyperbolic metric $g$. Under such metric, we will denote by $C_{(M,g)}$ the convex core of $(M,g)$ (minimal convex submanifold isotopic to $M$). The metric $g$ is said to be \textit{convex co-compact} if $C_{M}$ is compact. The space of convex co-compact metrics will de denoted by $QF(M)$, which by the Ahlfors-Bers theorem \cite{AhlforsBers} is homeomorphic to $\mathcal{T}(\partial M)$, the Teichm\"uller space of $\partial M$. If we would also like to consider parabolics, the appropriate property for $C_{M}$ is to have finite volume, in which case the metric is said to be \textit{geometrically finite}.

Given a convex co-compact hyperbolic 3-manifold $M$, Krasnov and Schlenker \cite{KrasnovSchlenker} defined its renormalized volume and calculated its first variation from the $W$-volume of a compact submanifold $N$ (as appears in [\cite{KrasnovSchlenker} Definition 3.1]) as
\begin{equation}
	W(M,N) = V(N) - \frac{1}{4}\int_{\partial N}Hda,
\end{equation}
where $da$ is the area form of the induced metric.

Expanding on the notation used in \cite{KrasnovSchlenker}, denote by $\I$ the metric induced on $\partial N$, $\II$ its second fundamental form (so $\II(x,y) = \I(x,By)$, where $B$ is the shape operator) and $\III(x,y) = \I(Bx,By)$ its third fundamental form.

If we further assume that $N$ has convex boundary and that the normal exponential map (pointing towards the exterior of $\partial N$) defines a family of equidistant surfaces $\lbrace S_r\rbrace$ that exhaust the complement of $N$ ($S_0= \partial N$), then the $W$-volume of $N_r$ (points on the interior of $S_r$) satisfies [\cite{Schlenker13} Lemma 3.6]
\begin{equation}\label{Wr}
	W(N_r) = W(N) - \pi r \chi(\partial N).
\end{equation}
Also, as observed in [\cite{Schlenker13} Definition 3.2, Proposition 3.3], $I^* =4 \displaystyle{\lim_{r\rightarrow\infty}e^{-2r}I_r}$ (where $I_r$ is the metric induced on $S_r$, which is identified with $S$ by the normal exponential map) exists and lies in the conformal class of the boundary. The analogous re-scaled limits for $\II, \III, B$ also exist and are denoted by $\II^*, \III^*, B^*$. 

For the case of convex co-compact manifolds, any metric $h$ at infinity that belongs to the conformal class given by the hyperbolic structure can be obtained as the rescaled limit of the induced metrics of some family of equidistant surfaces. Theorem 5.8 of \cite{KrasnovSchlenker} describes this by the use of Epstein surfaces (as stablished in \cite{Epstein}), which in turn allows us to define
\begin{equation}
	W(M,h) = W(M,N_r) + \pi r \chi(\partial M),
\end{equation}
where $\lbrace N_r\rbrace$ corresponds to the equidistant surfaces given by the Epstein surfaces of $h$. Then $W(M,h)$ is well-defined as a consequence of \eqref{Wr}.

We can finally define the renormalized volume of $M$ as
\begin{equation}
	V_R(M) = W(M,h),
\end{equation}
where $h$ is the metric in the conformal class at infinity that has constant curvature $-1$.

Krasnov and Schlenker \cite{KrasnovSchlenker} derived the variation formula of the $W$-volume in terms of the input at infinity (observe that because of the description by Epstein surfaces, $\I^*$ determines $\II^*$ and $\III^*$) from the volume variation of Rivin-Schlenker \cite{RivinSchlenker}.

As it is for example explained in \cite{Vargas15}, let us fix $c\in\mathcal{T}(\partial M)$ and some metric $\I_c$ that represents it. Then we have the variation of $V_R$ at $c$ [\cite{KrasnovSchlenker} Corollary 6.2, Lemma 8.5]
\begin{equation}
	DV_R(v) = -\frac{1}{4}\int_{\partial M}\langle D\mathrm{I}^*_c(v),\II^*_0 \rangle da^*,
\end{equation}
where the metric between tensors and the area form $da^*$ are defined from $\I_c$, $v\in RQ_c$ is the real part of a quadratic holomorphic differential with respect to $c$ and $\II^*_0 = \II^* - \frac{1}{2}\I^*$ is the traceless second fundamental form . This 2-form is (at each component of $\partial M$, after taking quotient by the action of $\pi_1(M)$) the negative of the real part of the Schwarzian derivative of the holomorphic map between one component of the region of discontinuity and a hyperbolic disk [\cite{KrasnovSchlenker} Lemma 8.3]. In particular (as explained in [\cite{Vargas15}, Section 2]) $\langle DI_c(v), \II_0\rangle = \langle v, \II_0\rangle$ pointwise. Then if we take $c$ to be a critical point (i.e.\ $DV_R(v)=0$ at $I_c$ for every $v\in RQ_c$) $\II_0$ must vanish at every point. This in turn implies that the holomorphic map between a component of the region of discontinuity and a disk has Schwarzian derivative identically zero, which means that the components are disks and the boundary of the convex core is totally geodesic.

There is an alternative approach to  $V_R$ defined in \cite{MoroianuGuillarmouRochon} that is really useful to prove continuity under limits, technique that we will adopt in Section \ref{sec:addlimit}. The equidistant foliation is interpreted by a function $\rho$ of $M$ (named \textit{geodesic boundary defining function}) satisfying:

\begin{equation}
\left|\frac{d\rho}{\rho}\right|_{g}^2=1, \quad (\rho^2g)|_{\partial M}=h^{\rm hyp}
\end{equation}

near the boundary of $M$, where $h^{\rm hyp}$ is the hyperbolic metric of $\partial M$ compatible with the conformal class at $\infty$. Then the renormalized volume of $M$ can be calculated as

\begin{equation}
	V_R(M)= \text{FP}_{z=0}\int_{M} \rho^z d\text{vol}.
\end{equation}

\cite{MoroianuGuillarmouRochon} also makes sense of this definition in the case where $(M,g)$ has parabolics.

\section{Geometrically finite limits}\label{sec:limits}
\begin{prop}\label{proplimit}
Let $M_n \in QF(M)$ be a sequence such that $V_C(M_n)$ is uniformly bounded. Then we can select finitely many base points such that (possibly after taking a subsequence)  $N_1,\ldots, N_k$ are the geometric limits corresponding those base points. These hyperbolic manifolds are geometrically finite and are obtained from $M$ by drilling curves into rank-$2$ cusps, pinching peripheral curves into rank-$1$ cusps or cutting essential cylinders forming one rank-$1$ cusp at each side. These cylinders are the division between distinct geometric limits $N_i$.
\end{prop}

\begin{proof}
Start recalling (see for instance \cite{Mardenbook}) that each boundary component of $C_{M_n}$ is path isometric to a surface of constant curvature, and say $S$ is one of such components. Then, by Deligne-Mumford compactification (see \cite{DeligneMumford} for the definition of the moduli stack $\mathcal{M}$, with compactifies Riemann surfaces by Riemann surfaces with nodes), we can assume that (after some relabelling) $S$ converges to a (possibly disconnected) union of hyperbolic surfaces of finite type $S_1,\ldots, S_l$, obtained by pinching some disjoint essential curves of $S$. Thus, since the path metric in $S$ bounds the metric in $M$, each component $S_j$ converges algebraically with the appropriately chosen base points (compare to the the proof of [\cite{BBCM16}, Theorem 2.8]). Since $S_j$ can not be elementary Riemann surface, select a closed geodesic and a basepoint on it. Do the same for all $j$ and all boundary components of $M$, and for each pair of basepoint sequences that stay at bounded distance, erase one of them. Each basepoint sequence gives a non-elementaty geometric limit, which we will label by $N_1,\ldots,N_k$.

Take $2\epsilon$ to be both a $2$ and $3$ dimensional Margulis constant. Then, since $V_C(M_n)$ is uniformly bounded, every component of the $\epsilon$-thick part of $C_{M_n}$, $C_{M_n}^\epsilon$, has finite diameter. Indeed, cover a component by an efficient cover of embedded balls of radius $2\epsilon$ (by efficient we mean that any two different centers are at least $2\epsilon$ apart). Then the balls with radius $\epsilon$ are disjoint, and either contained in $C_{M_n}$ or intersect its boundary. Since the volume of $C_{M_n}$ and the surface area of $\partial C_{M_n}$ are uniformly bounded (as well as for the $\epsilon$-neighbourhoods), the number of such balls is uniformly bounded as well. Each $N_i$ has associated one component of $C_{M_n}^\epsilon$. Moreover, following \cite{Mardenbook}(in particular Lema 4.3.1), if we pick base points other components of $C_{M_n}^\epsilon$ we can take a subsequence and assume that the geometric limit exists.

Consider any non-elementary geometric limit $N$ and assume that the limit basepoint lies on a closed geodesic of length greater than $2\epsilon$. Thanks to the uniform bounded diameter of the components of $C_{M_n}^\epsilon$, we know that all closed geodesics are at bounded distance from the basepoint, and from this we know that $N$ is geometrically finite. And if we take a neighborhood of $C_{N}^\epsilon$ (make $\epsilon$ small enough so it is connected) then it is the geometric limit set of $C_{M_n}^\epsilon$, or particularly of one the components that we considered from the start (with the appropriately chosen base points). So $\partial C_{N}$ corresponds to components of $\partial C_{M}$. Then the base point on the corresponding component $S_j \subseteq \partial C_M$ is at bounded distance, hence $N$ coincides with $N_i$ for some $1\leq i\leq k$.

From the previous paragraph we can also extrapolate that we can make $\epsilon$ smaller if necessary so $C_{N_i}^\epsilon$ is connected, and that there is a uniformly bounded amount of Margulis tubes that correspond only to cusps. This in turn says that different components of $C_{M_n}^\epsilon$ give different geometric limits since the distance between them goes to $\infty$.

The type of cusp at the limit can be described as follows: If the Margulis tube $T$ does not intersect the boundary of $C_{M_n}$, then corresponds to a rank-$2$ cusp. If $T$ intersects the boundary of $C_{M_n}$, will correspond to rank-$1$ cusp in one or many geometric limits $N_i$, depending on how it intersects the boundary. At every component $S$ of $\partial C_{M_n}$ intersected by $T$, the corresponding peripheral curve shrinks to a parabolic along the sequence (if not, from algebraic convergence of each subcomponent $S_1,\ldots,S_l$, the Margulis tube $T$ will eventually stay away from $S$). If $T$ only intersects one component $S$, the rank-$1$ cusp is created by pinching the peripheral representative at $S$, and $T$ does not separate different geometric limits $N_i$. If intersects multiple boundary components, then the different components of $\partial T \cap C_{M_n}$ are cylinders that locally separate geometric limits $N_i$. Each cylinder will correspond to a rank-$1$ cusp at a geometric limit $N_i$, and it will be called a \textit{cutting cylinder}.

Divide $M_n$ as follows: inside each Margulis tube $T$ draw a cylinder between its core geodesic and its peripheral geodesic representative at component of $\partial C_{M_n}$ intersected by $T$, and consider such cylinders to be disjoint. From the circles at $\partial C_{M_n}$ follow the normal geodesic flow to have a collection of $S^1\times \mathbb{R}$ cylinders, whose disjoint union we name by $\mathcal{C}$. Each component of $M_n \setminus \mathcal{C}$ has assigned a geometric limit $N_i$, so label these components as $M^i_n$. Inside each $M^i_n$ we have totally included $\epsilon$-Margulis tubes converging to rank-$2$ cusps (known as drilling) and the $\epsilon$-Margulis tubes intersecting $\partial C_{M_n}$ uniquely at $M^i_n$ converging to ranl-$1$ cusps (known as pinching).
\end{proof}

\begin{defi} We say that $N_1 \sqcup \ldots \sqcup N_k$ is the additive geometric limit of a sequence $\{M_n\}$ if they satisfy the properties described by Proposition \ref{proplimit}.
\end{defi}

Then an alternative formulation of \ref{proplimit} is that any sequence with bounded $V_C$ has a subsequence with a additive geometric limit. This notion of limit is related to Benjamini-Schramm convergence \cite{BenjaminiSchramm}. In their notion of limit, the basepoints are selected at random by a measure. Hence for our case, if we choose uniform measures on the convex cores, we will obtain $N_1 \sqcup \ldots \sqcup N_k$ with uniform measures on their convex cores, weighted out by their volumes.

\section{Peripherally generic fundamental domains}\label{sec:perigeneric}

For the present section, consider the projective Klein model of $\mathbb{H}^3$. Thus, $\mathbb{H}^3$ is represented by the unit ball in $\mathbb{P}^3$, $\partial\mathbb{H}^3$ by $S^2$, geodesic lines and planes are the intersection of euclidean lines and planes with $\mathbb{H}^3$. While this model is not conformal, orthogonality can be described by the pole and polar duality. This duality pairs points with planes, lines with lines, and can be extended from $\mathbb{R}^3$ to $\mathbb{P}^3$ as an algebraic map. Then a line $L$ and a plane $P$ are orthogonal if $P$ contains the polar of $L$. This is equivalent to $L$ containing the polar of $P$. 

 In the following section we will show that given a geometrically finite hyperbolic manifold $M$, its Dirichlet fundamental domain base at $p\in\overline{\mathbb{H}}^3$ will be \textit{peripherally generic} for almost all $p$. By peripherally generic we mean that (except when we refer to hyperplanes of the same abelian subgroup) no two different hyperplanes involved in the fundamental domain are tangent nor that their intersection is tanget to $S^2$, and that non three hyperplanes share a common line, as well as their triple intersection not to lie in $\partial\mathbb{H}^3$ (note that this is equivalent to the fundamental region at $S^2$ to be generic in the usual sense). While in \cite{Mardenbook} it is claim that generic Dirichlet fundamental regions are generic for most basepoints(not only peripherally generic), the proof in \cite{JorgensenMarden} is flawed as indicated in \cite{DiazUshijima}, where they solve the Fuchsian case. The techniques of the present section extrapolate their approach to the boundary of $M$.

Say then that $T$ is a isometry of $\mathbb{H}^3$. Define the $\textit{axis of } T$, $\text{Ax}(T)$, as the line between the two fixed points of $T$ at $S^2$ if $T$ is non parabolic, or the tanget line at the only fixed point of $T$ with direction tanget to any fiber of $T$, if $T$ parabolic. Observe then that $T$ sends $\text{Ax}(T)$ to itself in all cases, and since $T$ can be extended as a projective map from $\mathbb{P}^3$ to itself, then it must send the polar of $\text{Ax}(T)$, $\text{Ax}^*(T)$, to itself. This polar axis does not intersect the open unit ball, being only tangent to $S^2$ if $T$ is parabolic. Observe that unless $T$ is the identity or an involution, these are the only two lines preserved by $T$. One can also well define an isometry $\sqrt{T}$ such that has the same axis as $T$ and $\sqrt{T} \circ \sqrt{T} = T$.

For $p\in \mathbb{P}^3\setminus\text{Ax}^*(T)$ define the plane $P_{T,p}$ as the unique plane containing $\sqrt{T}(p)$ and $\text{Ax}^*(T)$. Observe that $T$ takes $P_{T^{-1},p}$ to $P_{T,p}$, and for $p\in\overline{\mathbb{H}}^3$, $p$ lies in between these planes. Then, if we define by $\mathcal{F}(L)$ to be the sheaf of planes containing a given line $L$, $P_{T,p}\in\mathcal{F}$ and the algebraic map $P_T:\mathbb{P}^3 \dashrightarrow \mathcal{F}(\text{Ax}^*(T))$ is not defined at $\mathbb{P}^3\setminus\text{Ax}^*(T)$. A geometric description for $P_T$ is the following. Take first $P$ is a tangent plane to $S^2$ containing $\text{Ax}^*(T)$, $P_T$ sends the points of $P$ to $P$ itself. Depending if $T$ is parabolic or not there could be one or two of these planes, but for planes in between $P_T$ uses $\text{Ax}^*(T)$ as an axis of rotation to find the target plane.

\begin{defi}
Let $T_1, T_2$ be two distinct isometries  of $\mathbb{H}^3$. Define $L_{T_1,T_2}(p) = P_{T_1}(p)\cap P_{T_2}(p)$ for $p\in\mathbb{P}^3$ such that $P_{T_1}(p), P_{T_2}(p)$ are defined and different.
\end{defi} 

From the definition of $P_T(.)$, we see that the codomain of $L_{T_1,T_2}$ should at least contain $\text{Ax}^*(T_1)\cup \text{Ax}^*(T_2)$, but could be greater is $P_{T_1}$ and $P_{T_2}$ are not generic for a given point. The following lemma answers this question.

\begin{lem}\label{codomainlema}
Let $T_1, T_2$ be two distinct isometries of $\mathbb{H}^3$ that don't generate a elliptic isometry. Then the codomain of $L_{T_1,T_2}$ is equal to:
\begin{enumerate}
	\item The tangent plane(s) to $S^2$ from $\text{Ax}^*(T_1)$, if $\text{Ax}^*(T_1) = \text{Ax}^*(T_2)$
	\item $\text{Ax}^*(T_1)\cup \text{Ax}^*(T_1)$, if $\text{Ax}^*(T_1)$ and $\text{Ax}^*(T_1)$ are not coplanar.
	\item $\text{Ax}^*(T_1)\cup \text{Ax}^*(T_1) \cup Q$, if $\text{Ax}^*(T_1)$ and $\text{Ax}^*(T_1)$ are distinct coplanar lines but the plane that contains them is not tangent to $S^2$. $Q$ is another coplanar line to them that contains $\text{Ax}^*(T_1)\cap\text{Ax}^*(T_2)$.
	\item $P$, where $P$ is the tangent plane to $S^2$ that contains $\text{Ax}^*(T_1)$ and $\text{Ax}^*(T_1)$.
\end{enumerate}
\end{lem}

\begin{proof}
For case (1), the polar axes coincide if and only if the axes coincide. Then the tangent plane(s) are send to themselves by $P_{T_1}$ and $P_{T_2}$, so the intersection will no be generic. For all other points, $P_{T_1}$ and $P_{T_2}$ only intersect at $\text{Ax}^*(T_1)$.

For case (2), we see that since there is not common plane containing them, the intersection $P_{T_1}\cap P_{T_2}$ is always a line.

For case (3), the points $p$ that will give us trouble will be the ones such that $P_{T_1}(p)$ and $P_{T_2}(p)$ are the common plane $P_0$. Since $P_0$ is not tanget to $S^2$, then $P_{T_1}^{-1}(P_0)$ is a plane distinct from $P_0$ that contains $\text{Ax}^*(T_1)$ (similarly for $T_2$). Then the set of trouble points is $P_{T_1}^{-1}(P_0) \cap P_{T_2}^{-1}(P_0) = Q$, which is a line since $P_{T_1}^{-1}(P_0), P_{T_2}^{-1}(P_0)$ are both distinct to $P_0$. The last property of $Q$ follow easily.

For case (4), similar to case (3), the extra codomain comes from $P_{T_1}^{-1}(P) \cap P_{T_2}^{-1}(P)$, but in this case the are both equal to $P$.

\end{proof}

Notice that in case $(4)$, $T_1$ and $T_2$ have a common fix point in $S^2$, which is the point of tangency of the plane containing them.

As in \cite{DiazUshijima}, $L_{T_1,T_2}$ is an algebraic map that could be extended to some of the points of its codomain as a map $\widehat{L}_{T_1,T_2}: \mathbb{P}^3 \dashrightarrow \mathcal{L}$, where $\mathcal{L}$ is the space of lines in $P^3$. What would be the new codomain and the nature of the extension is answered in the next lemma.

\begin{lem}
Let $T_1, T_2$ be two distinct isometries of $\mathbb{H}^3$that don't generate a elliptic isometry. Moreover, assume they are in cases $(1)-(3)$ of Lemma \ref{codomainlema}. Then the codomain of $\widehat{L}_{T_1,T_2}$ is equal to:
\begin{enumerate}\label{codomainlema2}
	\item $\emptyset$, if $\text{Ax}^*(T_1) = \text{Ax}^*(T_2)$. $\widehat{L}$ is constant equal to $\text{Ax}^*(T_1)$.
	\item $\text{Ax}^*(T_1)\cup \text{Ax}^*(T_2)$, if $\text{Ax}^*(T_1)$ and $\text{Ax}^*(T_1)$ are not coplanar.
	\item $\text{Ax}^*(T_1)\cup \text{Ax}^*(T_2) \cup Q$, if $\text{Ax}^*(T_1)$ and $\text{Ax}^*(T_1)$ are distinct coplanar lines but the plane that contains them is not tangent to $S^2$. $Q$ is another coplanar line to them as described in Lema \ref{codomainlema}.
\end{enumerate}
\end{lem} 
\begin{proof}
Let us look to each case of \ref{codomainlema}:
\begin{enumerate}
	\item Obvious since the map was already constant.
	\item Notice that while approaching a point in $\text{Ax}^*(T_1)$ you can obtain any possible plane containing $\text{Ax}^*(T_1)$, but $P_{T_2}$ converge to a unique plane not containing $\text{Ax}^*(T_1)$. Then the map $\mathcal{L}$ cannot be extended for points in $\text{Ax}^*(T_1)$, and similarly for $\text{Ax}^*(T_2)$.
	\item Since the plane containing $\text{Ax}^*(T_1),\text{Ax}^*(T_2)$ is not tangent to $S^2$ (hence not constant under $T_1,T_2$), the same argument as in the previous case proofs that $\mathcal{L}$ cannot be extende to $\text{Ax}^*(T_1)\cup\text{Ax}^*(T_2)$. And as for the line $Q$, approaching from the plane generated by $(Q,\text{Ax}^*(T_1))$ will conclude that the extension needs to be $\text{Ax}^*(T_2)$ but then we will have the same by swaping the indices $1,2$, which makes the extension impossible.
\end{enumerate}
\end{proof}

\begin{defi}
Let $T_1,T_2,T_3$ distinct non-elliptic isometries of $\mathbb{H}^3$. We denote by $\mathcal{A}_{T_1,T_2,T_3}$ the subset:
\[\mathcal{A}_{T_1,T_2,T_3}= \{p\in \mathbb{P}^3 | L_{T_1,T_2}(p)= L_{T_1,T_3}(p)\} \cup \text{codomain}(L_{T_1,T_2}) \cup \text{codomain}(L_{T_1,T_3})
\]
\end{defi}

It is not hard to see as in \cite{DiazUshijima}

\begin{lem}\label{Aalgebraic}
The set $\mathcal{A}_{T_1,T_2,T_3}$ is an algebraic subset of $\mathbb{P}^3$.
\end{lem}

Similarly, we can define the set of basepoints that give non peripherally generic bisecting planes

\begin{defi}
Let $T_1,T_2,T_3$ distinct non-elliptic isometries of $\mathbb{H}^3$. We denote by $\mathcal{D}_{T_1,T_2}$, $\mathcal{D}_{T_1,T_2,T_3}$ the subsets:
\[\mathcal{D}_{T_1,T_2} = \{p\in\mathbb{P}^3| L_{T_1,T_2}(p) \text{ is tangent to } S^2 \} \cup \text{codomain}(L_{T_1,T_2})
\]

\[\mathcal{D}_{T_1,T_2,T_3} = \{p\in\mathbb{P}^3\setminus \mathcal{A}_{T_1,T_2,T_3} |  (L_{T_1,T_2}(p)\cap L_{T_1,T_3}(p))\in S^2 \} \cup \mathcal{A}_{T_1,T_2,T_3}
\]
\end{defi}

The following lemma follows straight out of the definition.

\begin{lem}\label{Dalgebraic}
The sets $\mathcal{D}_{T_1,T_2},\mathcal{D}_{T_1,T_2,T_3}$ are algebraic subsets of $\mathbb{P}^3$.
\end{lem}

Since we are interested in cases when this algebraic sets are proper, let us describes the cases where they are not in the following lemmas.

\begin{lem}\label{lemaD2proper}
$\mathcal{D}_{T_1,T_2} = \mathbb{P}^3$ if and only if $\text{Ax}^*(T_1) = \text{Ax}^*(T_2)$ is tangent to $S^2$. In particular $T_1,T_2$ are both parabolic with the same axis.
\end{lem}
\begin{proof}
From Lema [\ref{codomainlema}] we see that for cases (2), (3) and (4) there is a point $p\in\mathbb{P}^3$ such that both bisecting lines are exterior to $S^2$, then $L_{T_1,T_2}(p) $ is not tanget to $S^2$, so $\mathcal{D}_{T_1,T_2}$ is a proper algebraic subset. For case (1) clearly needs to be as described.
\end{proof}

From now on consider every pair of isometries of $\mathbb{H}^3$ NOT in case $(4)$ of Lemma \ref{codomainlema}.

\begin{lem}\label{lemaAproper}
$\mathcal{A}_{T_1,T_2,T_3} = \mathbb{P}^3$ only if $\text{Ax}^*(T_1),\text{Ax}^*(T_2),\text{Ax}^*(T_3)$ all coincide or if pairwise they are in cases (3) of Lema [\ref{codomainlema2}]. For case (3), $\text{Ax}^*(T_3) = Q$.
\end{lem}
\begin{proof}
Indeed, the codomains of $\widehat{L}_{T_1,T_2}, \widehat{L}_{T_1,T_3}$ need to coincide since the algebraic functions are equal in a open set.
\end{proof}

\begin{lem}\label{lemaD3proper}
$\mathcal{D}_{T_1,T_2,T_3} = \mathbb{P}^3$ only if $\text{Ax}^*(T_1),\text{Ax}^*(T_2),\text{Ax}^*(T_3)$ coincide or if pairwise they are in cases (3) of Lema [\ref{codomainlema2}]. For case (3), $\text{Ax}^*(T_3) = Q$.
\end{lem}

\begin{proof}
Indeed, in the cases where $\mathcal{A}_{T_1,T_2,T_3} \neq \mathbb{P}^3$ and with generic basepoint outside the unit ball, the intersection point lies also outside the unit ball.
\end{proof}

\begin{defi}
Let $M$ be a geometrically finite hyperbolic $3$-manifold. We say that a Dirichlet fundamental domain $F_p$ with center $p$ is peripherally generic if its intersection with $S^2$ is a union of generic cuspid polygons (disregard the rank-$2$ cusps). Here generic means that two consecutive edges intersect transversally unless they join at a rank-$1$ cusp, and a vertex is not shared by more than $2$ edges.
\end{defi}

\begin{prop}\label{propgeneric}
Let $M$ be a geometrically finite hyperbolic $3$-manifold. Then for a dense set of $p\in\mathbb{H}^3$ the fundamental Dirichelet domain $F_p$ is peripherally generic.
\end{prop}

\begin{proof}
Denote by $\Gamma = \pi_1(M)<PSL(2,\mathbb{C})$. Define $\mathcal{D}$ as the union of all sets $\mathcal{D}_{T_1,T_2}$ for all pairs $T_1,T_2 \in \Gamma$ except when they are parabolic elements with the same fixed point, and all sets $\mathcal{D}_{T_1,T_2,T_3}$ for triples $T_1,T_2,T_3 \in \Gamma$, except when all three belong to the same abelian subgroup of $\Gamma$. Consider $p\notin\mathcal{D}$. The proof (similar to \cite{JorgensenMarden} \cite{DiazUshijima}) divides into showing that $\mathcal{D}_{T_1,T_2},\mathcal{D}_{T_1,T_2,T_3}$ are proper algebraic subsets (for the cases considered), and that this suffices for the fundamental polyhedron to be peripherally generic. The statement of the proposition follows from Baire's theorem and this two facts.

For $\mathcal{D}_{T_1,T_2}$ sets, we see clearly from Lema \ref{lemaD2proper} that this set is proper. For $\mathcal{D}_{T_1,T_2,T_3}$ sets, observes first that the isometries of $\Gamma$ with a common fix point form an abelian subgroup (because $\Gamma$ is discrete). Then the case $(4)$ of Lemma \ref{codomainlema} does not occur in our considerations. According to Lema \ref{lemaD3proper}, we need to discard cases $(3)$. For case (3) $T_2^2$ will fix two planes that are not the tangents from $Ax^*(T_2)$, implying that $T_2$ is elliptic.

Remains to show that for a point $p\notin\mathcal{D}$, $F_p$ is peripherally generic. Assume the contrary. Then $F_p$ has a polygonal face $\Pi$ in $S^2$ that is not a generic cuspid polygon. By the definitions of $\mathcal{D}_{T_1,T_2}, \mathcal{D}_{T_1,T_2,T_3}$ and $p\notin\mathcal{D}$, non-generic edge intersections ain $\Pi$ are not from the cases considered in $\mathcal{D}$. Let us look into each case.

\begin{itemize}
	\item Two consecutives edges of $\Pi$ are tangent: then the corresponding elements $T_1, T_2\in\Gamma$ associated to the edges must be parabolic with the same dual axis. Since they are appearing consecutively at a fundamental domain, they correspond to a rank-$1$ cusp, which is accounted in the definition of peripheral generic.
	\item A vertex of $\Pi$ is shared by three edges: then the corresponding elements $T_1.T_2,T_3\in\Gamma$ associated to the edges must belong to the same abelian subgroup. The subgroup cannot be loxodromic since in that case all the bisecting planes intersect at the common dual axis, which lies outside $S^2$. In the case that the subgroup is parabolic, the vertex will be then the common fix point. There can't be three edges for a rank-$1$ cusp and rank-$2$ cusp dixed point do not appear in $F_p$.
	
\end{itemize}

Then $F_p$ is peripherally generic, which is the last part of the proof.
\end{proof}

\section{Small deformations of sequences}\label{sec:smalldef}

As in \cite{Thurston5}, let us understand a hyperbolic $3$-manifold $M$ as a $({\rm PSL}(2,\mathbb{C}),\mathbb{H}^3)$ manifold, with associated holonomy $H:\pi_1(M) \rightarrow {\rm PSL}(2,\mathbb{C})$. For $M$ with finitely generated fundamental group (i.e. the interior of a compact $3$ manifold, thanks to the Scott/Shalen compact core \cite{Scott73}) the representation variety $\text{Def}(M) = Hom(\pi_1(M), {\rm PSL}(2,\mathbb{C}))/{\rm PSL}(2,\mathbb{C})$ is a finite dimensional complex algebraic variety.

\begin{defi}
Let $M$ be a geometrically finite hyperbolic $3$-manifold. Then define
\[
\text{Def}_0(M):= \{ H_0\in\text{Def}(M)|\text{ if }\gamma\text{ is conjugated to a rank-1 cusp},  H_0(\gamma) \text{ is parabolic} \}
\]
\end{defi}

As in Theorem 5.6 of \cite{Thurston5}, we have:

\begin{prop} The dimension of $\text{Def}_0(M)$ (near $H$) is as great as the total dimension of the Teichm\"uller space of $\partial(M)$, that is, 
\[{\rm dim}_\mathbb{C}(\text{Def}_0(M)) \geq \sum_{\chi(\partial M)_i < 0} 3|\chi(\partial M)_i| + (\text{number of rank-2 cusps})
\]
\end{prop}

\begin{defi}(Dehn-filling notation)
Let $M$ be a $3$-manifold with $k$ tori boundary components. For fixed $\alpha_i, \beta_i$ meridian and longitude of the $i$-tori component, and $(d_1,\ldots,d_k)\in S^2\times\ldots\times S^2$ ($S^2= \mathbb{R}^2\cup\{\infty\}$), denote by $M_{(d_1,\ldots,d_k)}$ as the result of gluing disks along the curves $a_i\alpha_i + b_i\beta_i$ ($d_i= a_i + b_i$) and then filling by $3$-balls. In case $d_i=\infty$, we do not perform any filling at the $i$-tori component.
\end{defi}

Following \cite{Thurston5},\cite{BonahonOtal88} let us show:

\begin{theorem}[Generalized Dehn-filling]\label{Dehnfill} Let $M$ be a geometrically finite hyperbolic $3$-manifold. Fix $\alpha_i,\beta_i$ meridians and longitudes  for the tori components of $\partial M$, and denote by $\tau\in\mathcal{T}(\partial M)$ the conformal class at infinity. Then exists a neighbourhood $U$ of $(\tau,\infty,\ldots,\infty)$ in $\mathcal{T}(\partial M)\times S^2\times\ldots \times S^2$ such that for $(\tau_0,d_1,\ldots,d_k)\in U$ there exists a hyperbolic structure in $M_{(d_1,\ldots,d_k)}$ with conformal class at infinity $\tau_0$, which we will call by $M_{(\tau_0,d_1,\ldots,d_k)}$. Moreover, as $(\tau_0,d_1,\ldots,d_k) \rightarrow (\tau,\infty,\ldots,\infty)$, $M_{(\tau_0,d_1,\ldots,d_k)}$ converges geometrically to $M$.

\end{theorem}

\begin{proof}
Thanks to Proposition \ref{propgeneric} take a fundamental domain $F_p$ that is peripherally generic. Then for elements $H_0\in\text{Def}_0(M)$ sufficiently close to $H$ (the holonomy of $M$), the elements that shape the cusped polygons for $F_p$ also shape generic cuspid polygons under this deformation (note that rank-$1$ cusps stay parabolic). Then by developing these cuspid polygons we have projective structures for each component of $\partial M$, which we can restrict to an element $\tau(H_0)\in\mathcal{T}(\partial M)$. Define then the map:
\begin{equation}\label{Tmap}
T:U\rightarrow \mathcal{T}(\partial M)\times \mathbb{C}^k, T(H_0) = (\tau(H_0),Tr(H_0(\alpha_1)^2,\ldots,Tr(H_0(\alpha_k)^2))
\end{equation}
were $U\subset \text{Def}_0(M)$ is the sufficiently small neighbourhood of $H$. Notice that $T$ is a holomorphic map were $\text{dim}_\mathbb{C}(U) \geq \text{dim}_\mathbb{C}(\mathcal{T}(M)) + k$, were this last expression is the dimension of the range of $T$ in (\ref{Tmap}).

We claim that $T^{-1}(\tau,4,\ldots,4) = \{H\}$. Indeed, $Tr(H_0(\alpha_i))=\pm2$ are the equations for $\alpha_i$ to be parabolic. Then the deformation space is characterized entirely by the conformal structure at infinity, concluding our claim.

Hence the dimensions at (\ref{Tmap}) must coincide and $T$ is an open map. As in Theorem 5.8.2, for $(\tau_0,d_1,\ldots,d_k)$ close to $(\tau,\infty,\ldots,\infty)$ there is a hyperbolic structure in $M_{(d_1,\ldots,d_k)}$ with holonomy map $H_0$ that satisfies $\tau(H_0)=\tau_0$. Since this manifold is hyperbolic, $\tau_0$ coincides with the conformal class at infinity. To close the argument, as $(\tau_0,d_1,\ldots,d_k) \rightarrow (\tau,\infty,\ldots,\infty)$, $H_0\rightarrow H$. Since $M$ is discrete, this garanties that $M_{(\tau_0,d_1,\ldots,d_k)}$ converges geometrically to $M$.
\end{proof}

\begin{prop}\label{smalldef}
Let $M\in QF(M)$ be a sequence with additive geometric limit $N_1 \sqcup \ldots \sqcup N_k$ (as in Propostion \ref{proplimit}). Then for any sufficiently small deformation of $N_1,\ldots,N_k$ there exist a sequence $\widehat{M}_n\in QF(M)$ that has them as additive geometric limit.
\end{prop}

\begin{proof}
Fix one component of the additive geometric limit $N_i$. Each rank-$2$ cusp in $N_i$ is the limit of a Margulis tube that stays inside $C_{M_n}$. Then choosing $\alpha_j,\beta_j$ meridian and longitude for the $j$ rank-$2$ cusp of $N_i$, the filling of this tube at $M_n$ is given by $d_{j,n}=(a_{j,n},b_{j,n})$. Hence $d_{j,n}\rightarrow\infty$ and $(N_i)_{d_{1,n},\ldots,d_{J,n}}$ is homeomorphic to $M^i_n$. Then applying Theorem \ref{Dehnfill}, $(N_i)_{\tau,d_{1,n},\ldots,d_{J,n}}$ is a sequence of hyperbolic manifolds homeomorphic to $M^i_n$, with only rank-$1$ cusps and geometric limit $(N_i,\tau)$. Let us call this sequence $(N_i)_{\tau,n}$ for simplicity.

Now, $(N_1)_{\tau,n}\sqcup\ldots\sqcup(N_i)_{\tau,n}$ have pairings between some of their rank-$1$ cusps corresponding to cutting cylinders. Perform then Klein-Maskit combinations between these rank-$1$ cusps(see [\cite{Maskit88}, Chapter 7] for more details in Klein-Maskit combination theory) with smaller and smaller horodisks, and in the glued manifold (which is homeomorphic to $M$ as one can observe easily) close all rank-$1$ cusp by self-Klein-Maskit combinations to obtain a sequence of hyperbolic manifolds $(M_{\tau,n})_m$ with the following properties:

\begin{itemize}
	\item $(M_{\tau,n})_m$ is homeomorphic to $M$ minus some curves parallel to the boundary. These rank-$2$ cusps are the only parabolics subgroups.
	\item $(N_1)_{\tau,n}\sqcup\ldots\sqcup(N_i)_{\tau,n}$ are all the possible geometric limits as $m\rightarrow\infty$ and we move the base point around.
\end{itemize}

Use again Theorem \ref{Dehnfill} to fill-in the rank-$2$ cusps of $(M_{\tau,n})_m$ by generalized Dehn-filling. These gives manifolds arbitrarily close to $(M_{\tau,n})_m$, so by the diagonal argument we can name these parabolic free manifolds by $(M_{\tau,n})_m$ and still have that $(N_1)_{\tau,n}\sqcup\ldots\sqcup(N_i)_{\tau,n}$ are the additive geometric limit as $m\rightarrow\infty$ and we move the base point around. In order to see that these manifolds are from the topological type desired, the original rank-$2$ cusps are filled as they were in the original sequence, while for the additional cusps created from Klein-Maskit combinations we can select fillings that don't change the topology. Indeed, since these cusps are parallel to the boundary, we can pick coefficients $d=(a,b)\in S^2$ such that the filled manifold is homeomorphic to $M$. Take then, for a tori component, $\alpha$ to be the curve isotopic to the boundary and $\beta$ a longitude. Then it is an easy $3$-manifold exercise to see that the coefficients $d=(n,1)$ fill-in a manifold homeomorphic to $M$.

By doing the diagonal argument one more time, we have constructed a sequence $M_{\tau,n}$ of hyperbolic manifolds homeomorphic to $M$ such that $(N_i,\tau)$ are the additive geometric limit.

\end{proof}

\section{Additive continuity of the renormalized volume}\label{sec:addlimit}

\begin{theorem}\label{continuitythm}
Let $M$ be a convex co-compact hyperbolic manifold with $\partial M \neq \emptyset$ incompressible. Let $M_n \in QF(M)$ be a sequence such that $V_\text{R}(M_n)$ converges. Then we can select finite many base points such that (possibly after taking a subsequence) $N_1,\ldots, N_k$ are the geometric limits corresponding to the base points (in the sense of Proposition \ref{proplimit}) and
\begin{equation}
	\lim_{n\rightarrow\infty} V_R(M_n) = \sum_{i=1}^{k} V_R(N_i)
\end{equation}
\end{theorem}

\begin{proof}
First, because of [\cite{BridgemanCanary15}, Theorem 1.1], we have that $V_\text{C}(M_n)$ is uniformly bounded. Then take into account the results of Proposition \ref{proplimit} as well as the definition for $M^i_n$. Then this theorem will follow from

\begin{equation}\label{mainlimit}
	\lim_{n\rightarrow\infty} \text{FP}_{z=0}\int_{M^i_n} \rho^z d\text{vol} = V_\text{R}(N_i)= \text{FP}_{z=0}\int_{N_i} \rho^z d\text{vol}
\end{equation}

since

\begin{equation}
	\sum^k_{i=1}\text{FP}_{z=0}\int_{M^i_n} \rho^z d\text{vol} = V_\text{R}(M_n),
\end{equation}

In order to show \ref{mainlimit} we will adapt \cite{MoroianuGuillarmouRochon}. The broad idea is to select neighbourhoods around the rank-$1$ cusps of $N_i$ and show the convergence of the integral at each cusp neighbourhood and on the complement of all of them. We will be using [\cite{MoroianuGuillarmouRochon}, Proposition 5.1] and [\cite{MoroianuGuillarmouRochon}, Corollary 5.3 ] in several steps of the proof, which are technical results for convergence of conformal factors after developing cusps in surfaces, which also applies in our situation. Let us examine then each possible scenario.

\textbf{Near a rank-$1$ cusp obtained by pinching:} Select the $\epsilon$ thin part of $C_{M_n}$ corresponding to the limit rank-$1$ cusp $\mathcal{C}$. Then take the region obtained by normal exterior geodesics to $\partial C_{M_n}$ as the neighbourhood $U^\mathcal{C}_n$ to analize in $M_n$. Given the geometric convergence (i.e. select basepoints in $M_n$ for the possible geometric limit corresponding to $\mathcal{C}$, giving all these sets parametrizations in $\mathbb{H}^3$), $U^\mathcal{C}_n$ converges to $U^\mathcal{C}$, the $\epsilon$ thin part corresponding to the rank-$1$ cusp. In order to calculate $V_R$ in these neighbourhoods, \cite{MoroianuGuillarmouRochon} parametrizes $U^\mathcal{C}_n, U^\mathcal{C}$ as follows.

Start considering the half upper-space model for $\mathbb{H}^3$. For a loxodromic transformation $\gamma$ with multiplier $e^{\ell(1+i\nu)}$ and fixed points $p,q\in \overline{\mathbb{C}}$, consider the flow lines of $\gamma^t$ for $t\in\mathbb{R}$. Consider as well the hyperplanes between $p$ and $q$ that are image of half-spheres centered at $0$ under a transformation that sends $\{0,\infty\}$ to $\{p,q\}$. Use the stereographic projection to parametrize the half-sphere of radius $1$ at $0$ (and hence also the corresponding hyperplane) by $\{z=v+iu, u>0\}$. Since the flow lines identify the hyperplanes with one another, we can parametrize $\mathbb{H}^3$ by $(w,\zeta)$, where $w=\frac{t}{2}$, $\zeta= z\ell$ and $(0,z)$ is our first parametrized hyperplane. In these variables, $\gamma$ sends $(w,\zeta)$ to $(w+\frac{1}{2},\zeta)$, $\{(w,i\ell), w\in\mathbb{R}\}$ parametrize the geodesic joining $p$ and $q$ and $\{-\frac{1}{4}\leq w \leq \frac{1}{4}, \zeta\}$ parametrizes the fundamental region for $\gamma$.

As in [\cite{MoroianuGuillarmouRochon}, Section 4] we have a neighbourhood of the Margulis tube in $N_i$ isometric to a neighbourhood of the manifold $(\mathbb{R}/\frac{1}{2}\mathbb{Z})_w\times \mathbb{H}^2_{\zeta=v+iu}$ equipped with the metric

\begin{equation}\label{model1gl}
\begin{gathered}
g_\ell=  \frac{{du}^2+{dv}^2+((1+\nu^2)R^4-4\nu^2\ell^2u^2)dw^2
+ 2\nu(R^2-2u^2)dwdv+4\nu uv dudw}{u^2}\end{gathered} 
\end{equation}
where $\exp^{\ell(1+i\nu)}$ is the multiplying factor of the geodesic of $M_n$ converging to a parabolic, and $R:=\sqrt{u^2+v^2+\ell^2}$. These neighbourhoods are intersections with sufficiently (but uniformly) small half hemispheres in $\mathbb{H}^3$ and then take quotient.

The limit model at $N_i$ is the manifold $(\mathbb{R}/\frac{1}{2}\mathbb{Z})_w\times \mathbb{H}^2_{\zeta=v+iu}$ equipped with the metric 

\begin{equation}\label{model1g0}
\begin{gathered}
g=  \frac{{du}^2+{dv}^2+(1+\nu^2)(u^2+v^2)^2dw^2+ 2\nu(v^2-u^2)dwdv+4\nu uv dudw}{u^2}\end{gathered}. 
\end{equation}

\textbf{Near a rank-$1$ cusp obtained by cutting cylinders:} Similarly as before, we have an isometry with a neighbourhood of $((\mathbb{R}/\frac{1}{2}\mathbb{Z})_w\times \mathbb{H}^2_{\zeta=v+iu},g_L)$, except that in this case we also need to cut along the corresponding cylinder $C$. Since again we have the convergence of $U^\mathcal{C}_\ell$ to $U^\mathcal{C}$, the region carved out by the cylinder $C$ is contained in smaller and smaller balls around $u=v=0$. Let us redefine then the cylinders $C$ such that the new cutting cylinders have a friendly description in $(w,v+iu)$ coordinates. For a fix $w$, take the lines (hyperbolic lines in $\overline{\mathbb{H}^2}$) joining $i\ell$ with $a_\ell$ and $i\ell$ with $b_\ell$, denoted by $\overline{i\ell,a_\ell}$ and $\overline{i\ell,b_\ell}$ where $b_\ell<a_\ell$. In order to choose $a_\ell, b_\ell$, observe that since all the region carved out by $C$ collapses at $u=v=0$, we can make a choice of $a_\ell,b_\ell$ for a given side of $C$. If we now see the other geometric limit adjacent to $C$ at the side of $a_\ell$, the boundary of the new cylinder is also of the type $\{-\frac{1}{4} \leq w \leq \frac{1}{4} \}\times\overline{i\ell,\hat{a}_\ell}$ in the coordinates of these geometric limit. This follows from the definition of the coordinates in terms of the loxodromic transformation $\gamma$. Then, we had selected one of $a_\ell,b_\ell$ for the adjacent geometric limit, so we make an arbitrary choice for the remaining of $a_\ell,b_\ell$ and then move to the next adjacent geometric limit. The process ends when, after moving cyclically around the geodesic corresponding to $\gamma$, we arrive to the final possible geometric limit.

Then in this case $U^\mathcal{C}_n$ is the neighbourhood of $u=v=0$ described in the pinching case minus $C_\ell:=\{(w,v+iu)| -\frac{1}{4} \leq w \leq \frac{1}{4}, \zeta \in \Delta_\ell \}$, where $\Delta_\ell$ is the region between $\overline{i\ell,a_\ell}$ and $\overline{i\ell,b_\ell}$.

  The picture one should have in mind is Figure \ref{fig:fig0}.

\begin{figure}
  \includegraphics[width=\linewidth]{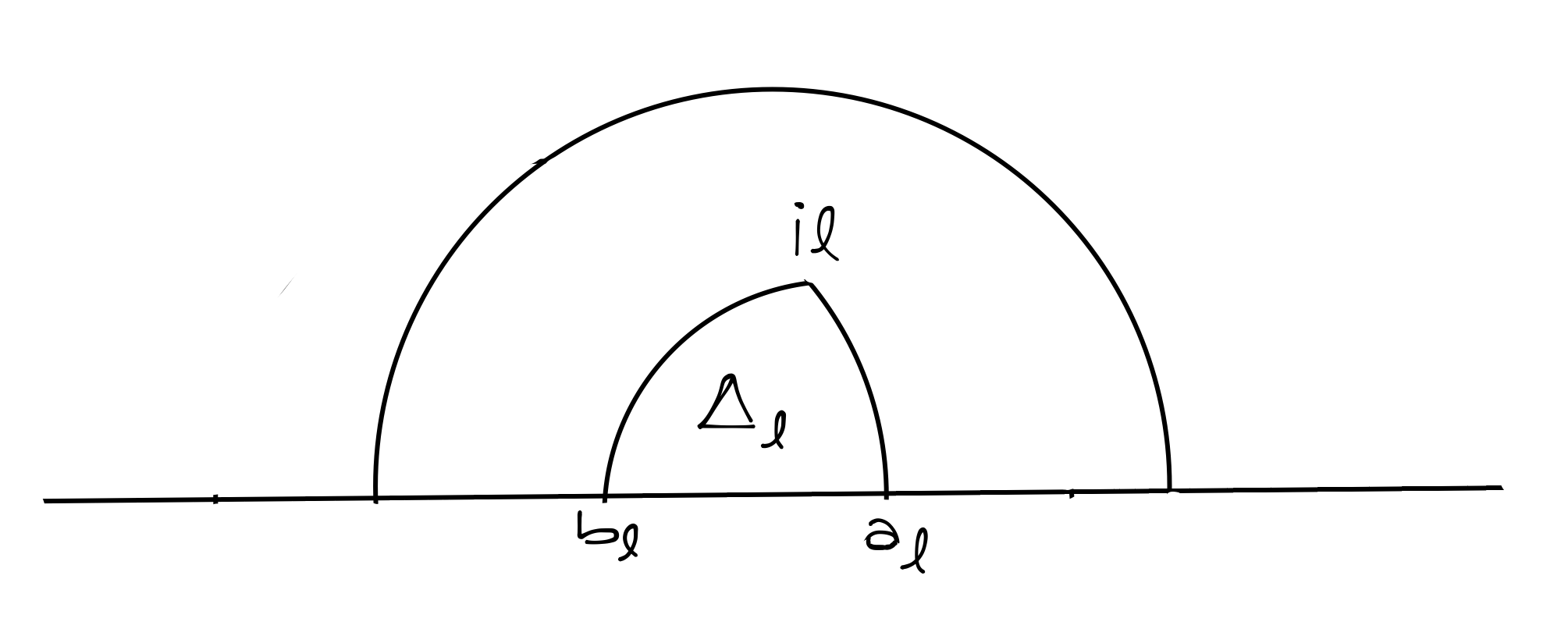}
  \caption{$\Delta_\ell$}
  \label{fig:fig0}
\end{figure}

The neighbourhood (prequotient) is the set between the halfspace previously mentioned and the pseudo-hyperplane obtained by lifting the cylinder $C$ to $\mathbb{H}^3$. This pseudo-hemisphere collapses to the parabolic fix point as $n\rightarrow\infty$, so the limit model for $N_i$ is the same as in the previous case (\ref{model1g0}).

\textbf{Limit far from the rank-$1$ cusps:} In this case the limit follows from showing convergence in compacts subsets of some geodesic boundary defining function $\rho_n$ of $M_n$ to the the geodesic boundary defining function $\rho$ of $N_i$. Say then that $\mathcal{K}$ is the complement of the rank-$1$ cusp neighbourhoods. Then $\rho$ is the solution of the Hamilton-Jacobi equation

\begin{equation}
\left|\frac{d\rho}{\rho}\right|_{g}^2=1, \quad (\rho^2g)|_{\partial M}=h^{\rm hyp}.
\end{equation}

Then we define an auxiliar geodesic boundary function $\widehat{\rho_n}$ as the solution of

\begin{equation}\label{hamjab0}  
\left|\frac{d\widehat{\rho}_n}{\widehat{\rho}_n}\right|_{g_n}^2=1, \quad \widehat{\omega}_n|_{\rho=0}=0
\end{equation}

where $\widehat{\rho_n} = e^{\widehat{\omega_n}}\rho$.

The result follows as done in \cite{MoroianuGuillarmouRochon}, so we will just cite their work. The main difference is that we have to apply to each component $M^i_n$ since our case doesn't have to be connected. Next is [\cite{MoroianuGuillarmouRochon}, Lemma 6.3]

\begin{lem}\label{lemhatomega}
There exists $\delta>0$ such that for sufficiently large $n$, the Hamilton-Jacobi equation \eqref{hamjab0} has a solution $\widehat{\omega}_n$ in 
$\mathcal{K}\cap \{\rho<\delta\}$ and $\widehat{\omega}_n$ converges to $0$ in $\mathcal{C}^k$-norms there for all $k$.
\end{lem}

The relationship between $\widehat{\rho}_n$ and $\rho_n$ is
\begin{equation}\label{rhoepsinK}
\rho_n=e^{\omega_n}\widehat{\rho}_n
\end{equation} 
where $\omega_n$ is the solution of 
\[\left|\frac{d\rho_n}{\rho_n}\right|_{g_{n}}^2=1, \quad \omega_n|_{\rho=0}=\varphi_n\]
and $\varphi_n$ is the uniformization factor such that $h_n^{\rm hyp}:=e^{2\varphi_n}h_n$ is hyperbolic if $h_n:=(\rho^2g_n)|_{\rho=0}$;
The Hamilton-Jacobi equation \eqref{rhoepsinK} has a unique solution in $\mathcal{K}$ near $M$ and in particular one has $\omega|_{\mathcal{K}\cap M}=\varphi=0$.

Then this result is used in order to show continuity far from rank-$1$ cusps.

\begin{prop}\label{limoutsidecusp}[\cite{MoroianuGuillarmouRochon}, Proposition 8.1]
Let $\rho_n\in\mathcal{C}^\infty(\overline{{M_n}})$ be a geodesic boundary defining function such that $h_n:=(\rho_n^2g_n)|_{\partial M_n}$ is the unique hyperbolic metric in the conformal boundary ($\rho_n$ is uniquely defined near $\partial M$). 
Let $\rho\in \mathcal{C}^\infty(\overline{N_i}_c)$ be a geodesic boundary defining function of $\overline{N_i}$ with $h:=(\rho^2g)|_{\partial{N_i}}$ being the unique finite volume hyperbolic metric in the conformal boundary ($\rho$ is uniquely defined near $\overline{N_i}$). 
Let $\theta^i_n$ be a family of smooth functions on $\overline{M^i_n}$ with support in $\overline{{M^i_n}}$ and converging in all $\mathcal{C}^k$-norms to $\theta_i$, a function in $\overline{N_i}$ that vanishes in a neighbourhood of the rank-$1$ cusps. The following limit holds 
\[ \lim_{n\to \infty}\Big({\rm FP}_{z=0} \int_{M^i_n}\theta^i_n \rho_n^{z}\, {\rm dvol}_{g_n}\Big)= 
{\rm FP}_{z=0} \int_{N_i}\theta_i \rho^{z}\, {\rm dvol}_{g}.\]
\end{prop}

We next study the behaviour of the renormalized volume in the regions containing the degeneration to rank $1$-cusps. We notice that the main Theorem \ref{continuitythm} follows from Proposition \ref{limoutsidecusp} and the following 
\begin{prop}\label{nearthecusp}
With the notations and assumptions of Proposition \ref{limoutsidecusp} and Theorem \ref{continuitythm}, we have
\[
\lim_{n\rightarrow\infty}\,  {\rm FP}_{z=0}\, \int_{M^i_n}(1-\theta^i_n)\rho^z_n{\rm dvol}_{g_n}=
{\rm FP}_{z=0}\, \int_{N_i}(1-\theta_i)\rho^z_0{\rm dvol}_{g_0}.
\]
\end{prop}
\begin{proof} We can assume that $(1-\theta^i_n)$ is supported in the region $U^\mathcal{C}_n$ for each rank-$1$ cusp $\mathcal{C}$, then we can assume that we have the parametrization $(w,\zeta=v+iu)$ of (\ref{model1gl}), where we have forgot the $n$ parameter and use rather $\ell$ with $\ell\to 0$, and $\nu=\nu(\ell)$ is converging to some limit $\nu_0$ as $\ell\to 0$. Then we are going to show equality at each rank-$1$ cusp appearing at $N_i$ from either a pinching curve or a cutting cylinder from $M^i_n$. The rest of the proof follows [\cite{MoroianuGuillarmouRochon}, Proposition 8.2] for both pinching and cutting cylinder.

First, using \ref{model1gl}, we can calculate ${\rm dvol}_{g_\ell}$ as
\[{\rm dvol}_{g_\ell}=\frac{R^2dudvdw}{u^3}\]
where $R^2=u^2+v^2+\ell^2$. Then the results follows from showing
\begin{equation}\label{renormint} 
\lim_{\ell\to 0}{\rm FP}_{z=0}\int_{(u,v,w,\ell)\in \overline{\mathcal{U}}_\ell} \rho_\ell^z \chi_\ell \frac{R^2dudvdw}{u^3}=
{\rm FP}_{z=0}\int\rho_0^z \chi_0 \frac{R^2dudvdw}{u^3}
\end{equation}
for the pinching case, and

\begin{equation}\label{renormintcyl} 
\lim_{\ell\to 0}{\rm FP}_{z=0}\int_{(u,v,w,\ell)\in \overline{\mathcal{U}}_\ell \setminus C_\ell} \rho_\ell^z \chi_\ell \frac{R^2dudvdw}{u^3}=
{\rm FP}_{z=0}\int\rho_0^z \chi_0\frac{R^2dudvdw}{u^3}
\end{equation}
for the cutting cylinder case. In both limits $\rho_\ell=\rho_n$ is the function solving 

\begin{equation}  \label{hamjab}
\left|\frac{d\rho_\ell}{\rho_\ell}\right|_{g_\ell}^2=1, \quad \rho_l=e^{\omega_\ell}U \text{ for some } \omega_\ell \text{ satisfying } (\omega_\ell)_{U=0}=\varphi_\ell
\end{equation}

with $e^{2\varphi_\ell}h_\ell$ being hyperbolic if $h_\ell$ is given by $h_\ell=\frac{g_\ell}{U^2}$ (here we are also introducing the notation $U = \frac{u}{R}$), and $\chi_\ell \in \mathcal{C}^\infty_c(\overline{\mathcal{U}}_\ell)$ is equal to $1$ near $u=v=0$ and converges to $\chi_0$. Note that for the convergence in \eqref{renormint}, \eqref{renormintcyl}, we can choose $\chi$ independent to $\ell$(as in \cite{MoroianuGuillarmouRochon}). We will study the convergence of \eqref{renormintcyl}, since the proof of \eqref{renormint} appears in \cite{MoroianuGuillarmouRochon}. Nevertheless, the limits to be calculated are analogous to one another, so the reader can follow both arguments simultaneously.

Let us start by dividing the integral \eqref{renormintcyl} as the sets $R_1(\ell), R_2(\ell), R_3(\ell)$ defined by:
\begin{align}
R_1(\ell)&= \{(w,\zeta=v+iu) \in U^\mathcal{C}_\ell \;|\; u\geq \vert v\vert,\; u^2+ v^2 \geq \ell^2\}\\
R_2(\ell)&= \{(w,\zeta=v+iu) \in U^\mathcal{C}_\ell \;|\; u^2 + v^2 \leq \ell^2\}\\
R_3(\ell)&= \{(w,\zeta=v+iu) \in U^\mathcal{C}_\ell \;|\; \vert v\vert\geq u,\; u^2+v^2 \geq \ell^2\}.
\end{align}
Compare with the proof of [\cite{MoroianuGuillarmouRochon}, Proposition 8.2] for the notation (see Figure \ref{fig:fig1}).

\begin{figure}
  \includegraphics[width=\linewidth]{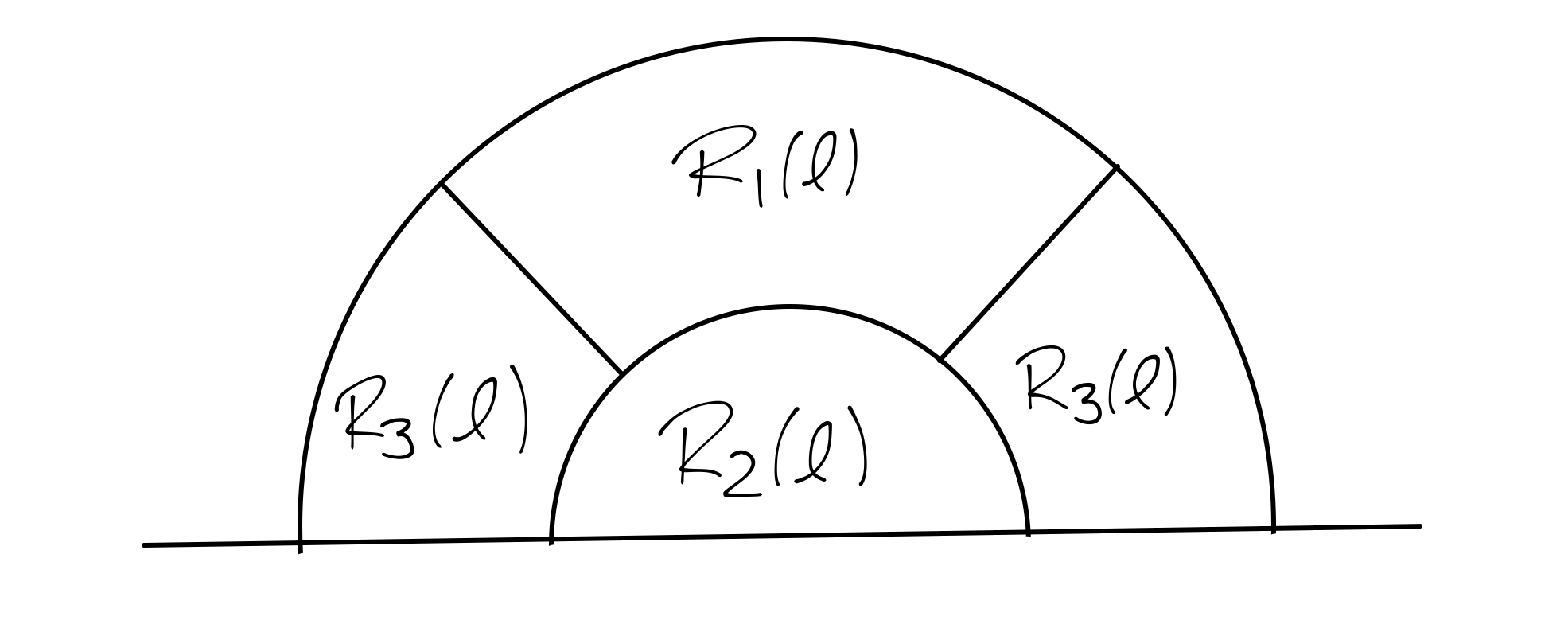}
  \caption{$w$-slice of $R_1(\ell), R_2(\ell), R_3(\ell)$}
  \label{fig:fig1}
\end{figure}

We can write the $R_1(\ell)$-term as

$$
R_1(\ell)= \{ (w,r,\theta) \; | \;  \frac{\pi}{2}\leq \theta \leq \frac{3\pi}{2}, \; \ell \leq r \leq 2\delta, \; -\frac{1}{4}\leq w \leq\frac{1}{4} \} 
$$
where we use the following coordinates,
\begin{equation}
  u = r\sin\theta, \quad v = r\cos\theta, \quad w.
\label{bdf.21}
\end{equation}

\begin{figure}
  \includegraphics[width=\linewidth]{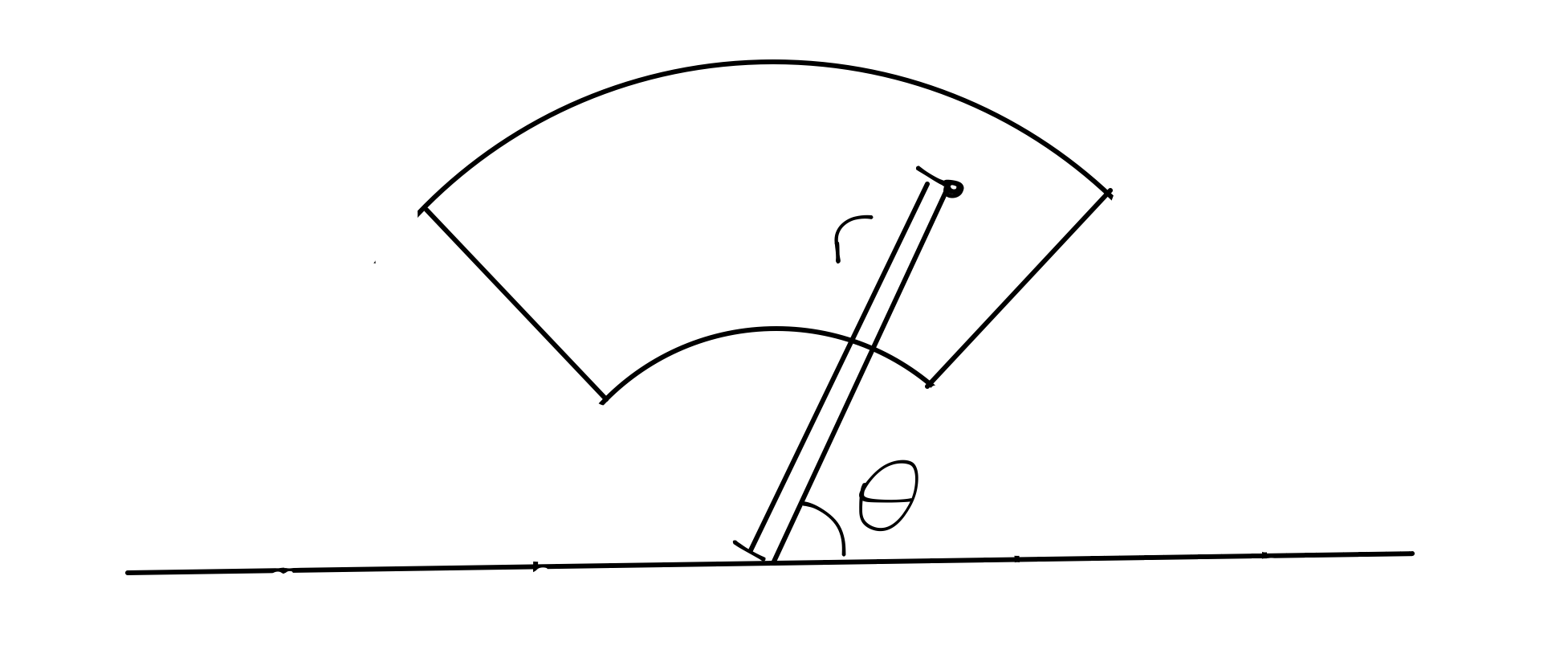}
  \caption{Coordinates of $R_1(\ell)$}
  \label{fig:fig2}
\end{figure}

Restricted to this region, $\int  \chi \frac{R^2dudvdw}{u^3} = \int \chi \frac{1 + (\ell^2/r^2) d\theta drdw}{\sin^3\theta}$ both integrals are finite and there is no need to renormalize.  Thus,
\[
\begin{aligned}
{\rm FP}_{z=0}\int_{R_1(\ell) \setminus C_\ell} \rho_\ell^z \chi \frac{R^2dudvdw}{u^3}&= \int_{R_1(\ell) \setminus C_\ell} \chi(r\sin\theta,r\cos\theta,w)
\left( 1+ \frac{\ell^2}{r^2} \right) \sin^{-3}\theta d\theta drdw
\end{aligned}
\]
where $R_1(\ell)=\{ (w,V,u) \;|\; -\frac{1}{4}\leq w\leq \frac{1}{4}, -1\leq V\leq 1, \ell\leq u \leq \delta\}$. We can use dominated convergence (recall that $\frac{\ell^2}{r^2} \leq 1$) to deduce 
that 
\begin{equation}\label{bdf.22}
\begin{split}
\lim_{\ell\to 0} {\rm FP}_{z=0}\int_{(w,r,\theta)\in R_1(\ell)\setminus C_\ell} \rho_\ell^z \chi \frac{1 + (\ell^2/r^2) d\theta drdw}{\sin^3\theta}=&
\int_{-\frac{1}{4}}^{-\frac{1}{4}}
\int_{\ell}^{2\delta}\int_{\frac{\pi}{2}}^{\frac{3\pi}{2}} \frac{\chi(r\sin\theta,r\cos\theta,w)}{\sin^3\theta} d\theta drdw\\
=&{\rm FP}_{z=0}\int_{R_1(0)} \rho_0^z \frac{\chi d\theta drdw}{\sin^3\theta}.
\end{split}
\end{equation}

Next, let's look to the region $R_2(\ell)$

\begin{equation}
R_2(\ell):= \left\{(w,v,u)\;|\; 0\leq u,\; u^2+v^2\leq \ell^2,\; -\frac14\leq w \leq\frac14 \right\}
\end{equation}

Define then the change of coordinates:

\begin{equation}
u=\frac{\ell\sin\theta(\ell^2+\hat{v}^2)}{\cos\theta(\ell^2-\hat{v}^2)+\ell^2+\hat{v}^2}, \quad v= \frac{2\cos\theta\ell^2\hat{v}}{\cos\theta(\ell^2-\hat{v}^2)+\ell^2+\hat{v}^2}.
\end{equation}

A geometric interpretation of these new variables is to parametrize $\overline{\mathbb{H}^2}_{v+iu}$ by the geodesics from $i\ell$ to points $\hat{v}\in\mathbb{R}\subset\mathbb{H}^2$, where for any of such geodesics, $\cos\theta$ is the euclidean distance from the origen, after we have identified $\mathbb{H}^2$ with the unit disk ($i\ell\leftrightarrow 0$). Hence 

\begin{equation}
\overline{\mathbb{H}^2}= \{v+iu \;|\; 0\leq u\} = \{(\hat{v},\theta) \;|\; 0\leq\theta\leq\frac{\pi}{2}\},
\end{equation}

where $i\ell$ is identified with the line $\mathbb{R}\times\{\frac{\pi}{2}\}$, $\mathbb{R}\times 0$ is identified with itself by the identity, and the geodesic from $i\ell$ to $\hat{v}$ is identified with the line $\{\hat{v}\}\times \left[0,\frac{\pi}{2}\right]$. Also, the Jacobian can be easily calculated as

\begin{equation}
\frac{\partial(v,u)}{\partial(\hat{v},\theta)} = \frac{2\cos\theta\ell^3(\ell^2+\hat{v}^2)}{(\cos\theta(\ell^2-\hat{v}^2)+\ell^2+\hat{v}^2)^2}
\end{equation}

In these new coordinates, the regions $R_2(\ell), C_\ell$ are defined by

\begin{align}
R_2(\ell):= \left\{(w,\hat{v},\theta)\;|\; 0\leq \theta\leq \frac{\pi}{2},\; \vert\hat{v}\vert\leq\ell,\; -\frac14\leq w \leq\frac14 \right\}\\
C_\ell:= \left\{(w,\hat{v},\theta)\;|\; 0\leq \theta\leq \frac{\pi}{2},\; b_\ell\leq\hat{v}\leq a_\ell,\; -\frac14\leq w \leq\frac14 \right\},
\end{align}

\begin{figure}
  \includegraphics[width=\linewidth]{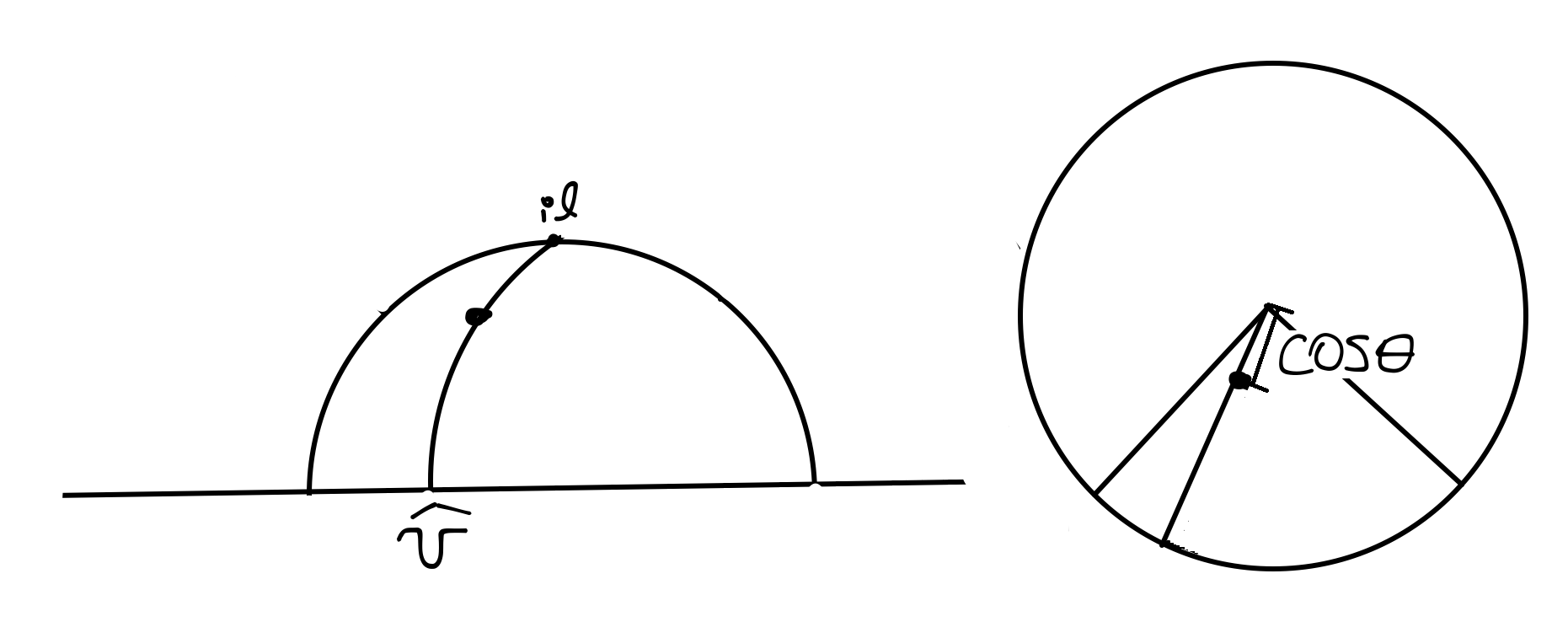}
  \caption{Coordinates of $R_2(\ell)$}
  \label{fig:fig3}
\end{figure}

The two limits we are looking at are:

\begin{equation}
\begin{split}
\lim_{\ell\to 0}{\rm FP}_{z=0}\int_{R_2(\ell) \setminus C_\ell} \rho_\ell^z \chi \frac{R^2dudvdw}{u^3} &=  \lim_{\ell\to 0}{\rm FP}_{z=0}\int_{R_2(\ell) \setminus C_\ell} \rho_\ell^z \chi \frac{2\ell^2\cos\theta d\theta d\hat{v}dw }{\ell^2+\hat{v}^2}\\
&= 0.
\end{split}
\end{equation}

\begin{equation}
\begin{split}
\lim_{\ell\to 0}{\rm FP}_{z=0}\int_{R_3(\ell) \setminus C_\ell} \rho_\ell^z \chi \frac{R^2dudvdw}{u^3} &=  \lim_{\ell\to 0}{\rm FP}_{z=0}\int_{R_3(\ell) \setminus C_\ell} \rho_\ell^z \chi \frac{2\ell^2\cos\theta d\theta d\hat{v}dw }{\ell^2+\hat{v}^2}\\
&= {\rm FP}_{z=0}\int_{R_3(0)} \rho_0^z \chi \frac{2\cos\theta d\theta d\hat{v}dw }{1+v^2}.
\end{split}
\end{equation}

Observe then that the following statements are sufficient to proof our result

\begin{equation}\label{R2lim}
\lim_{\ell\to 0}{\rm FP}_{z=0}\int_{-\frac14}^\frac14 \int_{\alpha_\ell}^{\beta_\ell} \int_0^{\frac{\pi}{2}} \rho_\ell^z \chi \frac{2\ell^2\cos\theta d\theta d\hat{v}dw }{\ell^2+\hat{v}^2} =0,
\end{equation}

for sequences $\vert\alpha_\ell\vert, \vert\beta_\ell\vert \leq \ell \leq \kappa_\ell$ all with limit $0$.

Rescale $\hat{V}=\frac{\hat{v}}{\ell}$ so now:

\[
{\rm FP}_{z=0}\int_{-\frac14}^\frac14 \int_{\kappa_\ell}^{2\delta} \int_0^{\frac{\pi}{2}} \rho_\ell^z \chi \frac{2\ell^2\cos\theta d\theta d\hat{v}dw }{\ell^2+v^2} = {\rm FP}_{z=0}\int_{-\frac14}^\frac14 \int_{\alpha_\ell/\ell}^{\beta_\ell/\ell} \int_0^{\frac{\pi}{2}} \rho_\ell^z \chi \frac{2\ell\cos\theta d\theta d\hat{V}dw }{1+\hat{V}^2}.
\]

As with $R_1(\ell)$, the integral is finite and dominated by 

\[
{\rm FP}_{z=0}\int_{-\frac14}^\frac14 \int_{-1}^{1} \int_0^{\frac{\pi}{2}} \frac{2\ell\cos\theta d\theta d\hat{V}dw }{1+\hat{V}^2} = \frac{\pi\ell}{2},
\]

so \eqref{R2lim} follows.

Finally, let us deal with the region $R_3(\ell)$. Similarly to the previous coordinates, we will parametrize $R_3(\ell)$ by geodesics joining $\tilde{v}$ and $i\ell$. Hence the coordinates $(\theta,\tilde{v})$ are defined by by

\[
v =  \frac{\tilde{v}}{2} - \frac{\ell^2}{2\tilde{v}} + \cos\theta\left( \frac{\tilde{v}}{2} + \frac{\ell^2}{2\tilde{v}}\right),\quad u = \sin\theta\left( \frac{\tilde{v}}{2} + \frac{\ell^2}{2\tilde{v}}\right)
\]

Here, $\theta$ is the counterclockwise angle on the half-circle joining $\tilde{v}$ and $i\ell$, as represented in Figure \ref{fig:fig4}.

\begin{figure}
  \includegraphics[width=\linewidth]{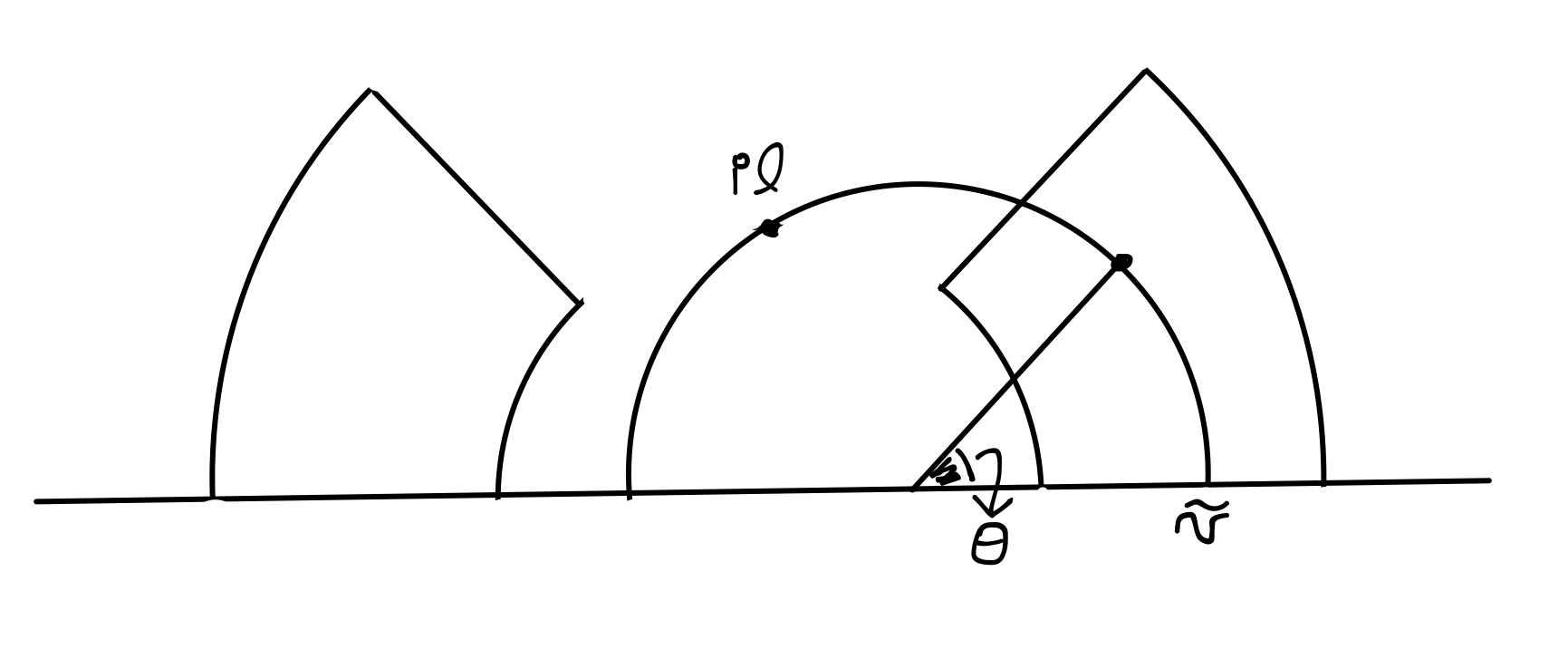}
  \caption{Coordinates of $R_3(\ell)$}
  \label{fig:fig4}
\end{figure}

The Jacobian of the change of variables can be calculated by

\[
\frac{\partial(v,u)}{\partial(\hat{v},\theta)} = \frac{\hat{v}}{4} \left( 1 + \frac{\ell^2}{\tilde{v}^2}\right) \left( 1 - \frac{\ell^2}{\tilde{v}^2} + \cos\theta\left( 1 + \frac{\ell^2}{\tilde{v}^2}\right)\right),
\]

and the representation of $R_3(\ell)$ as

\[
R_3(\ell) = \{ (w,\tilde{v},\theta) \;|\; 0\leq\theta\leq\frac{\pi}{2},\; \ell \leq \vert\tilde{v}\vert \leq 2\delta,\; -\frac14\leq w\leq \frac14 \}.
\]

Here there is actually an extra region already accounted by $R_1(\ell)$. Nevertheless, the result still follows while the notation is more simple.

Similar to the reasoning of \eqref{R2lim}, the desired limit will be

\begin{equation}\label{R3lim}
\begin{split}
&\lim_{\ell\to 0}{\rm FP}_{z=0}\int_{-\frac14}^\frac14 \int_{\kappa_\ell}^{2\delta} \int_0^{\frac{\pi}{2}} \rho_\ell^z \chi \frac{\left( 1 + \frac{\ell^2}{\tilde{v}^2} + \cos\theta\left( 1 - \frac{\ell^2}{\tilde{v}^2}\right)\right) \left( 1 - \frac{\ell^2}{\tilde{v}^2} + \cos\theta\left( 1 + \frac{\ell^2}{\tilde{v}^2}\right)\right)  d\theta d\tilde{v}dw}{\sin^3\theta \left( 1+ \frac{\ell^2}{\hat{v}^2}\right)} \\
&= {\rm FP}_{z=0}\int_{-\frac14}^\frac14 \int_{0}^{2\delta} \int_0^{\frac{\pi}{2}} \rho_0^z \chi \frac{\left(1 + \cos\theta \right)^2      d\theta d\tilde{v}dw}{\sin^3\theta}\\
\end{split},
\end{equation}
for some sequence $\kappa_\ell\geq\ell$ with limit equal to $0$.

With the notation of \eqref{hamjab}, $\rho_\ell = e^{\omega_\ell}\frac{u}{R} = e^{\omega_\ell}\sin\theta \sqrt{\frac{\left( 1+\frac{\ell^2}{\hat{v}^2}\right)}{2\left( 1 + \frac{\ell^2}{\hat{v}^2} + \cos\theta\left( 1 - \frac{\ell^2}{\hat{v}^2}\right)\right)}}$. And since $\ell\leq\hat{v}$, then $\left(\frac{\left( 1+\frac{\ell^2}{\hat{v}^2}\right)}{2\left( 1 + \frac{\ell^2}{\hat{v}^2} + \cos\theta\left( 1 - \frac{\ell^2}{\hat{v}^2}\right)\right)}\right) \leq \frac{1}{2}$ and $\omega_\ell$ has the expansion with respect to $\theta$ (as proved in [\cite{MoroianuGuillarmouRochon}, Proposition 6.7]):

\begin{equation}\label{omegaR3}
\omega_\ell = a_0 + a_2 \left(\frac{u}{R}\right)^2 + \mathcal{O}\left( \left(\frac{u}{R}\right)^3\right) = a_0 + a_2\sin^2\theta \left(\frac{\left( 1+\frac{\ell^2}{\hat{v}^2}\right)}{2\left( 1 + \frac{\ell^2}{\hat{v}^2} + \cos\theta\left( 1 - \frac{\ell^2}{\hat{v}^2}\right)\right)}\right) + \mathcal{O}\left( \theta^3\right),
\end{equation}

where $a_0, a_2$ depende on $\ell, \hat{v}$ and $w$ but not on $\theta$.

Then the finite part of \eqref{R3lim} can be decomposed as $I_1(\ell) + I_2(\ell)$, where:

\[
I_1(\ell) := {\rm FP}_{z=0}\int_{R_3(\ell)\setminus C_\ell} \chi\sin^z\theta \frac{\left( 1 + \frac{\ell^2}{\tilde{v}^2} + \cos\theta\left( 1 - \frac{\ell^2}{\tilde{v}^2}\right)\right) \left( 1 - \frac{\ell^2}{\tilde{v}^2} + \cos\theta\left( 1 + \frac{\ell^2}{\tilde{v}^2}\right)\right)  d\theta d\tilde{v}dw}{\sin^3\theta \left( 1+ \frac{\ell^2}{\hat{v}^2}\right)}  
\]
\[
I_2(\ell) := {\rm res}_{z=0}\int_{R_3(\ell)\setminus C_\ell}\chi\omega_\ell \sin^z\theta \frac{\left( 1 + \frac{\ell^2}{\tilde{v}^2} + \cos\theta\left( 1 - \frac{\ell^2}{\tilde{v}^2}\right)\right) \left( 1 - \frac{\ell^2}{\tilde{v}^2} + \cos\theta\left( 1 + \frac{\ell^2}{\tilde{v}^2}\right)\right)  d\theta d\tilde{v}dw}{\sin^3\theta \left( 1+ \frac{\ell^2}{\hat{v}^2}\right)}.
\]

Similar to the parallel case in the proof of [\cite{MoroianuGuillarmouRochon}, Prop. 8.2], we can observe that for $I_1(\ell)$

\[
I_1(\ell) = \int_{-\frac14}^{\frac14}\int_{\kappa_\ell}^{2\delta} q_1(\tilde{v}, \frac{\ell}{\tilde{v}}, w) d\tilde{v}dw
\]
for some smooth function $q_1$ independent from $\ell$. Then we can easily see that $\lim_{\ell\to 0} I_1(\ell) = I_1(0)$.

For $I_2(\ell)$, we can simplify by replacing $\omega_\ell$ with the first two terms of \eqref{omegaR3}. Then

\[
I_2(\ell) = \int_{-\frac14}^{\frac14}\int_{\kappa_\ell}^{2\delta} a_0 q_2(\tilde{v}, \frac{\ell}{\tilde{v}}, w) + a_2 q_3(\tilde{v},\frac{\ell}{\tilde{v}},w) d\tilde{v}dw,
\]
for some smooth functions $q_2, q_3$ independent from $\ell$, and $a_0, a_2$ given by (from [\cite{MoroianuGuillarmouRochon}, Proposition 6.7])

\[
a_0(w,\tilde{v})=\varphi_\ell(w,\tilde{v})
\],  
\[
a_2(w,\tilde{v}) = -\frac14 |d\varphi_{\ell}|^2_{h_{\ell}}+ \frac{C_1\ell^2+C_2v\partial_w\varphi_\ell}{(\ell^2+v^2)} +C_3v\partial_v\varphi_\ell +\frac12,
\]
where $C_1, C_2, C_3$ are constants smoothly depending on $\nu$. Then using [\cite{MoroianuGuillarmouRochon}, Proposition 5.1], [\cite{MoroianuGuillarmouRochon}, Corollary 5.3] for the integral convergence of $\varphi_\ell,d\varphi_\ell$ to $\varphi_0, d\varphi_0$, we can see that $\lim_{\ell\to 0} I_2(\ell) = I_2(0)$. To see that this concludes all the cases, see that the integral with limits $-2\delta \leq \tilde{v} \leq -\kappa_\ell$ follows by analogy and the integral with limits $\alpha_\ell \leq \tilde{v} \leq \beta_\ell$ converges to $0$ for $\alpha_\ell,\beta_\ell \rightarrow 0$ either both greater than $\ell$ or smaller than $-\ell$. Then the proof of Proposition \ref{nearthecusp} is finished.

\end{proof}
\end{proof}

\section{Consequences}\label{sec:consequences}

In order to describe the infimum of $V_R$, let us set some notation.

\begin{defi}\label{defacy} A pair $(M,P)$ is a paired acylindrical manifold if $M$ is a compact irreducible $3$-manifold and $P\subseteq \partial M$ is a collection of incompressible tori and annuli such that:
\begin{itemize}
	\item Every non-cyclic subgroup of $\pi_1(M)$ is conjugated to some component of $P$.
	\item Every essential cylinder in $(M,\partial M)$ is isotopic to a component of $P$.
\end{itemize}
\end{defi}

Compare to [\cite{Morgan84}, Definition 4.8]. Moreover, using that exact same chapter (more precisely Theorema A and B \cite{Morgan84}), we know that every paired manifold with $\partial M\neq\emptyset$ different from the unit ball is hyperbolizable, with a geometrically finite metric and $P$ corresponding to the parabolic locus. Likewise, we can also use the term acylindrical for a (possibly cusped) hyperbolic manifold.

\begin{defi} A geometrically finite hyperbolic manifold $N$ is said acylindrical if $(N,P)$ is an acylindrical paired manifold, where $P$ is the parabolic locus of $N$ in $\partial N$.
\end{defi}

Then, as in \cite{Morgan84}, if $N$ is an acylindrical hyperbolic manifold then there is a hyperbolic metric in $N$ with the same parabolic locus which has totally geodesic boundary.

\begin{theorem}\label{mainconsequence}
Let $M$ be a hyperbolizable compact $3$-manifold where $\partial M\neq\infty$ is incompressible and has no tori boundary components. Then
\[\inf V_\text{R}(M)= \frac{v_3}{2} \Vert DM \Vert
\]
where $v_3$ is the volume of the regular ideal tetrahedron in $\mathbb{H}^3$, $DM$ is the double of the manifold $M$ and $\Vert\cdot\Vert$ denotes the Gromov norm of a manifold. Moreover, for any sequence $\{M_n\}$ such that $\lim_{n\rightarrow\infty}V_\text{R}(M_n)=\inf V_\text{R}(M)$, there exist a decomposition of $M$ along essential cutting cylinders in components $A_1\sqcup \ldots \sqcup A_s\sqcup F_1\sqcup \ldots\sqcup F_r$ (with $A_1,\ldots,A_s$ acylindrical and $F_1,\ldots,F_r$ fuchsian) such that $A_1 \sqcup \ldots\sqcup A_s\sqcup F_1\sqcup \ldots\sqcup F_r$ is the additive geometric limit of a subsequence of $\{M_n\}$.

\end{theorem}

\begin{proof}
Indeed, because of Proposition \ref{proplimit} and Theorem \ref{continuitythm}, any sequence $M_n\in QF(M)$ where $\lim_{n\rightarrow\infty}V_\text{R}(M_n)=\inf V_\text{R}(M)$ has a subsequence with additive geometric limit $N_1\sqcup \ldots \sqcup N_k$ and 
\[\inf V_\text{R}(M) = \lim_{n\rightarrow\infty}V_\text{R}(M_n) = \sum_{i=1}^{k} V_R(N_i).\] 
Because of Proposition \ref{smalldef}, any small deformation of $N_1 \sqcup\ldots\sqcup N_k$ is the additive geometric limit of another sequence $\widehat{M}_n\in QF(M)$. Then $N_1,\ldots,N_k$ are critical points for $V_R$, which implies that their convex cores have totally geodesic boundaries. Hence each $N_i$ is either acylindrical or fuchsian (depending if the convex core has non-empty interior or not), although the acylindrical components could arrive from pinching, drilling and cutting cylinders instead of just the later as stated in the theorem. The next step is to notice that pinching and drilling increases $V_\text{R}$. so let us assume that at least one curve gets pinched or drilled while converging to $N_1,\ldots,N_k$.

Consider $DN_i$, the double of the manifolds $N_i$. Each acylindrical component doubles into a finite volume hyperbolice manifold by doubling along the geodesic boundary of $C_(N_i)$. Each fuchsian component doubles into a Seifert fibered manifold, in fact as the product of a finite type surface $S$ with $S^1$. The cusps from $N_i$ give rank-$2$ cusp in the following pattern:

\begin{itemize}
	\item A rank-$2$ cusp (which is only obtained by drilling) gives two rank-$2$ cusps in $DN_i$, one per copy of $C_{N_i}$ in $DN_i$.
	\item A rank-$1$ cusp (obtained by either pinching or a cutting cylinder) gives one rank-$2$ cusp in $DN_i$.
\end{itemize}

Moreover, if we glue $N_1,\ldots,N_k$ along paired rank-$1$ cusps (paired by cutting cylinders) we obtain a manifold $M^*$ that topologically is $M$ minus the drilled curves. Then we can glue $DN_1,\ldots,DN_k$ along paired cusps (again, paired by cutting cylinders) to obtain $DM^*$, which is $DM$ minus some curves (two for each drilled curve and one for each pinched curve). Then $DN_1,\ldots,DN_k$ is a decomposition of $DM^*$ along incompressible tori into finite-volume hyperbolic manifold or Seifert fibered manifolds, where now cutting cylinders can be seen as cutting tori and each component is (finite-volume) hyperbolic or Seifert fibered. Given Gromov's theorem (as seen in [\cite{Thurston6}, Theorem 6.2]) we can relate the renormalized volume to the Gromov norm $\Vert \cdot\Vert$:

\begin{equation}\label{VRGromov}
V_R(N_i)=\frac{v_3}{2}\Vert DN_i\Vert 
\end{equation}

where $v_3$ is the volume of the regular ideal tetrahedron in $\mathbb{H}^3$. Indeed, when $N_i$ is acylindrical, $V_R(N_i)$ is half the hyperbolic volume of $DN_i$, which is equal to $v_3\Vert DN_i\Vert$. When $N_i$ is Fuchsian both sides of the equality vanish.

Now, for fixed large $n$, consider $M_i$ equal to the double of $M^i_n$ along $\partial_0 M^i_n$. Noticing that cutting cylinders are glued into tori, we can paste along those tori to obtain $DM$ from $M_1,\ldots,M_k$. Now, we can divide each $M^i_n$ by essential cylinders until each component forms an acylindrical pair with the mentioned cylinders, hence hyperbolizable with totally geodesic convex core by the discussion after Definition \ref{defacy}. As before, these essential cylinders double into essential tori in $M_i$ that divide it into components that are either finite volume hyperbolic or Seifert fibered, depending if the convex core of the corresponding component had empty interior or not. Hence we have a decomposition of $M$ as in the statement of this theorem, so we will label the components as we did there. The decomposition $A_1,\ldots,A_s,F_1,\ldots,F_r$ is a subdecomposition of $M^1_n,\ldots,M^k_n$

From [\cite{Thurston6}, Proposition 6.5.2] and [\cite{Thurston6}, Theorem 6.5.6], since $M_i$ can be obtained from $N_i$ by filling some cusps, we have

\begin{equation}\label{Gromovineq}
\Vert N_i\Vert \geq \Vert M_i \Vert
\end{equation}

where the inequality is strict if we fill at least one cusp.

Also, by applying [\cite{Thurston6}, Proposition 6.5.2,] and [\cite{Thurston6}, Theorem 6.5.5] to each $M_i$ and then add them up, we have

\begin{equation}\label{Gromovhyp}
\sum_{i=1}^{k} \Vert M_i \Vert = \sum_{j=1}^{s} \Vert A_j \Vert = \Vert DM \Vert
\end{equation}

Putting (\ref{VRGromov}),(\ref{Gromovineq}) and (\ref{Gromovhyp}) together and recalling that at least one curve was pinched or drilled

\begin{equation}
\sum_{i=1}^{k} V_R(N_i) > \frac{v_3}{2} \Vert DM \Vert
\end{equation}

We will then contradict that $N_1,\ldots,N_k$ was obtained as the infimum sequence with at least one curve being pinched or drilled as soon as we observe that  $A_1,\ldots,A_s,F_1,\ldots,F_r$ can be also obtained as limit. Indeed, as in the proof of Proposition \ref{smalldef}, by doing Klein-Maskit combinations and generalized hyperbolic Dehn-fillings we can obtain a sequence of geometrically finite hyperbolic manifolds homeomorphic to $M$ with limit $A_1,\ldots,A_s,F_1,\ldots,F_r$.

\end{proof}

From this we can easily see the following corollary for quasifuchsian manifolds

\begin{cor} Let $M$ be a quasifuchsian manifold. Then $V_\text{R}(M)\geq 0$ with equality if and only if $M$ is Fuchsian.
\end{cor}

Moreover, any sequence such that $V_\text{R}\rightarrow 0$  must converge to a disjoint union of Fuchsian manifolds, since there cannot be a non-zero volume in Theorem \ref{mainconsequence}.

Also, since for a acylindrical manifold $M$ there cannot be cutting cylinders, there is only one possible geometric limit under our conditions. Hence from Theorem \ref{mainconsequence} we have (also proved in \cite{Vargas16})

\begin{cor} Let $M$ be a acylindrical hyperbolizable $3$-manifold. Then any sequence $M_n\in QF(M)$ such that $\lim_{n\rightarrow\infty}V_\text{R}(M_n)=\inf V_\text{R}(M)$ converges geometrically to $M_{tg}\in QF(M)$, the metric with convex core totally geodesic.  
\end{cor}

\bibliographystyle{amsalpha}
\bibliography{mybib}

\end{document}